\documentclass[11pt]{amsart}
\usepackage{graphicx}
\usepackage{amssymb}
\usepackage{amsmath}
\usepackage{amsthm,amsfonts,bbm}
\usepackage{amscd}
\usepackage{geometry}
\usepackage[all,2cell]{xy}
\usepackage{epsfig,epstopdf}
\usepackage[backref]{hyperref}
\usepackage{tikz}
\usepackage{caption}
\usepackage{subcaption}
\UseAllTwocells \SilentMatrices
\newtheorem{thm}{Theorem}[section]
\newtheorem{cor}[thm]{Corollary}
\newtheorem{lem}[thm]{Lemma}

\theoremstyle{definition}
\newtheorem{defi}[thm]{Definition}
\theoremstyle{remark}

\numberwithin{equation}{section}
\numberwithin{figure}{section}
\def \cl{\text{Cl}}
\def\deg{\text{deg}}
\def \d{\delta}

\def \deg{\text{deg}}
\def \dim{\text{dim}}

\def \im{\text{im~}}

\def \L{\mathcal{L}}
\def \lk{\text{Lk}}

\def \lamax{\lambda_{\max}}

\def \p{\partial}

\def \R{\mathbb{R}}
\def \s{\sigma}
\def \st{\text{St}}

\DeclareMathOperator{\sgn}{sgn}

\begin{document}
\title[Largest eigenvalue and balancedness of complex]{The largest normalized Laplacian eigenvalue and incidence balancedness of simplicial complexes}

\author[Y.-M. Song]{Yi-min Song}
\address{Center for Pure Mathematics, School of Mathematical Sciences, Anhui University, Hefei 230601, P. R. China}
\email{songym@stu.ahu.edu.cn}

\author[H.-F. Wu]{Hui-Feng Wu}
\address{Center for Pure Mathematics, School of Mathematical Sciences, Anhui University, Hefei 230601, P. R. China}
\email{wuhf@stu.ahu.edu.cn}

\author[Y.-Z. Fan]{Yi-Zheng Fan*}
\address{Center for Pure Mathematics, School of Mathematical Sciences, Anhui University, Hefei 230601, P. R. China}
\email{fanyz@ahu.edu.cn}
\thanks{*The corresponding author.
This work was supported by National Natural Science Foundation of China (Grant No. 12331012).}



\begin{abstract}
Let $K$ be a simplicial complex, and let $\Delta_i^{up}(K)$ be the $i$-th up normalized Laplacian of $K$.
Horak and Jost showed that the largest eigenvalue of $\Delta_i^{up}(K)$ is at most $i+2$, and characterized the equality case by the orientable or non-orientable circuits.
In this paper, by using the balancedness of signed graphs, we show that
$\Delta_i^{up}(K)$ has an eigenvalue $i+2$ if and only if $K$ has an $(i+1)$-path connected component $K'$ such that the $i$-th signed incidence graph $B_i(K')$ is balanced, which implies Horak and Jost's characterization.
We also characterize the multiplicity of $i+2$ as an eigenvalue of $\Delta_i^{up}(K)$, which generalizes the corresponding result in graph case.
Finally we gave some classes of infinitely many simplicial complexes $K$ with $\Delta_i^{up}(K)$ having an eigenvalue $i+2$ by using wedge, Cartesian product and duplication of motifs.
\end{abstract}

\maketitle

\section{Introduction}
The Laplacian operator on graphs and simplicial complexes has a long and rich history. In 1847 Kirchhoff \cite{Kirch} formulated the celebrated matrix-tree theorem with an application in electric networks.
In the early 1970s Fiedler \cite{Fied} used the  second smallest Laplacian eigenvalue to characterize the connectivity of a graph.
Bottema \cite{Bot} introduced the normalized graph Laplacian for the study  of transition probability on graphs; also see \cite{Chung93} for more details.
 Eckmann \cite{Eckmann} extended the Laplace operator from graphs to simplicial complexes, providing  the discrete version of the Hodge theorem, which can be expressed as
$ \text{ker}(\d_i^* \d_i+ \d_{i-1} \d_{i-1}^*) = \tilde{H}^i(K,\mathbb{R}),$
where $\d_i^* \d_i+ \d_{i-1} \d_{i-1}^*$ is the higher order combinatorial Laplacian of the simplicial complex $K$.
Duval and Reiner \cite{Duv} show that the combinatorial Laplace operator of a shifted simplicial complex has an integral spectrum.
Dong and Wachs \cite{Dong} gave an elegant proof of Bouc's result \cite{Bouc} on the decomposition of the representation of the symmetric group on the homology of a matching complex, and proved the spectrum of the Laplace operator of the matching complex is also integral.
   Some remarkable recent results on combinatorial Laplacian for simplicial complexes involve lots of algebraic and geometric aspects for complexes, which are applied in the study of independent number, chromatic number, theta number \cite{Bachoc, Golubev,HJ13B}.

Horak and Jost \cite{HJ13B} developed a general framework for Laplace operators defined in terms of the combinatorial structure of a simplicial complex, including the combinatorial Laplacian and the normalized Laplacian.
They also presented some upper bounds for a general Laplacian of simplicial complex $K$ in the paper.
Let $K$ be a simplicial complex and let  $L_i^{up}(K), \Delta_i^{up}(K)$ be the $i$-th up combinatorial  Laplacian  and normalized Laplacian of $K$, respectively.
For the combinatorial Laplacian, Horak and Jost \cite{HJ13B} showed that
$$\lamax(L_i^{up}(K)) \le (i+2) \max_{F \in S_i(K)} \deg F.$$
This bound was improved by Fan, Wu and Wang \cite{FWW} as
$$ \lamax(L_i^{up}(K)) \le \max_{\bar{F} \in S_{i+1}(K)} \sum_{F \in \p \bar{F}} \deg F.$$
Duval and Reiner \cite{Duv} proved that
$$\lamax(L_i^{up}(K))  \le n,$$
where $n$ is the number of vertices of $K$.

For the normalized Laplacian, Horak and Jost \cite{HJ13B} proved the eigenvalues of $\Delta_i^{up}(K)$ lie in the interval $[0,i+2]$, and characterized the simplicial complex $K$ with $\Delta_i^{up}(K)$ having an eigenvalue $i+2$ as follows.

\begin{thm}\cite[Theorem 7.2]{HJ13B}\label{jost}
For an $(i+1)$-path connected simplicial complex $K$, the following statements are equivalent:

$(1)$ the spectrum of $\Delta_i^{up}(K)$ contains the eigenvalue $i + 2$;

$(2)$ there are no $(i + 1)$-orientable circuits of odd length nor $(i + 1)$-non orientable circuits of even length in $K$.
\end{thm}

When $i=0$ (i.e. graph case), a $1$-circuit is exactly a cycle which is always orientable.
So, Theorem \ref{jost} implies that a connected graph has an normalized Laplacian eigenvalue $2$ if and only if it contains no odd cycle (or equivalently, $G$ is bipartite); also see \cite[Chapter 1]{Chung97}.
However, it may be difficult to check the condition (2) in  Theorem \ref{jost} for a general $i$.

In this paper we will use the balancedness of the $i$-th signed incidence graph to give an equivalent characterization of the condition (1)  in  Theorem \ref{jost}.
Our characterization is only depending on the $i$-th signed incidence  graph of $K$, and is independent of the orientation of faces of $K$ (Theorem \ref{main-result}).
We also characterized the multiplicity of $i+2$ as an eigenvalue of $\Delta_i^{up}(K)$, which generalizes the corresponding result on graph case (see \cite[Chapter 1]{Chung97}).
Finally we gave some classes of infinitely many simplicial complexes $K$ with $\Delta_i^{up}(K)$ having an eigenvalue $i+2$ by using wedge, Cartesian product and duplication of motifs.

We remark that the signed incidence graph plays an important role in the Laplacian spectrum of simplicial complexes.
Fan, Wu and Wang \cite{FWW} proved that the largest eigenvalue of the $i$-th up Laplacian of a simplicial complex $K$ is not greater than the largest eigenvalue of the $i$-th up signless Laplacian of $K$. Moreover, if $K$ is $(i+1)$-path connected, then the equality holds if and only if the $i$-th signed incidence graph $B_i(K)$ of $K$ is balanced.

\section{Preliminaries}
\subsection{Simplical complex and Laplace operator}
Let $V$ be a finite set.
An \emph{abstract simplicial complex} (simply called \emph{complex}) $K$ over $V$ is a collection of the subsets of $V$ which is closed under inclusion.
An \emph{$i$-face} or an \emph{$i$-simplex} of $K$ is an element of $K$ with cardinality $i+1$.
The \emph{dimension} of an $i$-face is $i$, and the dimension of $K$ is the maximum dimension of all faces of $K$.
The faces which are maximum under inclusion are called \emph{facets}.
We say $K$ is \emph{pure} if all facets have the same dimension.

Let $S_i(K)$ be the set of all $i$-faces of $K$ for $i\ge 0$.
The $p$-skeleton of $K$, written $K^{(p)}$, is the set of all simplices of $K$ of dimension less than or equal to $p$.
So, $K^{(1)}$ is the usual graph with vertex set $V(K)$ consisting of $0$-faces usually called \emph{vertices}, and  edge set $E(K)$ consisting of $1$-faces usually called \emph{edges}.
We say $K$ is \emph{connected} if the graph $K^{(1)}$ is connected.
An \emph{$i$-path} of length $m$ in $K$ is an ordering of $i$-simplices: $F_1,F_2,\cdots,F_m$, such that $F_i \cap F_j$ is an $(i-1)$-face of $K$ if and only if $|j-i|=1$.
When $F_m$ coincides with $F_1$, we say that $L$ is an \emph{$i$-circuit} of length $m-1$.
The complex $K$ is \emph{$i$-path connected} if any two $i$-faces $F_1$ and $F_2$ of $K$ are connected by an $i$-path.
An \emph{$i$-path connected component} of $K$ is a maximal $i$-path connected sub-complex of $K$ under inclusion.

We say a face $F$ is \emph{oriented} if we assign an ordering of its vertices and write it as $[F]$.
Two ordering of the vertices of $F$ are said to determine the \emph{same orientation} if there is an even permutation transforming one ordering into the other.
If the permutation is odd, then the orientation are opposite.
The \emph{$i$-chain group} of $K$ over $\mathbb{R}$, denoted by $C_i(K,\mathbb{R})$, is the vector space over $\R$ generated by all oriented $i$-faces of $K$ modulo the relation $[F_1]+[F_2]=0$, where $[F_1]$ and $[F_2]$ are two different orientations of a same face.
The \emph{cochain group} $C^i(K,\R)$ is defined to be the dual of $C_i(K,\mathbb{R})$, i.e. $C^i(K,\R)=\text{Hom}(C_i(K,\mathbb{R}),\R)$, which are generated by the dual basis consisting of $[F]^*$ for all $F \in S_i(K)$, where
$ [F]^*([F])=1, [F]^*([F'])=0 \text{~for~} F' \ne F.$
The functions $[F]^*$ are called the \emph{elementary cochains}.

We may assume that $\emptyset \in K$, called the empty simplex with dimension $-1$.
Then $C_{-1}(K, \R)=\R \emptyset$, identified with $\R$,
and $C^{-1}(K, \R)=\R \emptyset^*$, also can be identified with $\R$, where $\emptyset^*$ is the identify function on the empty simplex.
For each integer $i=0,1,\ldots,\dim K$, the \emph{boundary map} $\p_i: C_i(K,\R) \to C_{i-1}(K,\R)$ is defined to be
$$\p_i([v_0,\ldots,v_i])=\sum_{j=0}^i (-1)^j[v_0,\ldots,\hat{v}_j,\ldots,v_i],$$
for each oriented $i$-face $[v_0,\ldots,v_i]$ of $K$,
where $\hat{v}_j$ denotes the vertex $v_j$ has been omitted.
We will have the \emph{augmented chain complex} of $K$:
$$  \cdots \longrightarrow C_{i+1}(K,\R) \stackrel{\p_{i+1} }{\longrightarrow} C_i(K,\R) \stackrel{\p_{i} }{\longrightarrow} C_{i-1}(K,\R) \longrightarrow \cdots \longrightarrow  C_{-1}(K,\R) \longrightarrow 0,$$
satisfying $\p_i \circ \p_{i+1}=0$.
The \emph{$i$-th reduced homology group} of $K$ is defined to be $\tilde{H}_i(K)=\ker \p_i / \im \p_{i+1}$.
The complex $K$ is called \emph{acyclic} if its reduced homology group vanishes in all dimensions.

Here, by abuse of notation, we use $\p \bar{F}$ to denote the set of all $i$-faces in the boundary of $\bar{F}$ when $\bar{F} \in S_{i+1}(K)$.
If $[\bar{F}]:=[v_0,\ldots,v_i]$ and $[F_j]:=[v_0,\ldots,\hat{v}_j,\ldots,v_i]$,
then we define $\sgn([F_j], \p[\bar{F}])=(-1)^j$, namely, the sign of $[F_j]$ appeared in $\p[\bar{F}]$, and $\sgn([F], \p[\bar{F}])=0$ if $F \notin \p \bar{F}$.
The \emph{coboundary map} $\delta_{i-1}: C^{i-1}(K,\R) \to C^i(K,\R)$ is the conjugate of $\p_i$ such that $ \d_{i-1} f = f \p_i.$
So
$$ (\d_{i-1} f)([v_0,\ldots,v_i])=\sum_{j=0}^i (-1)^jf([v_0,\ldots,\hat{v}_j,\ldots,v_i]).$$
Similarly, we have the \emph{augmented cochain complex} of $K$:
$$  \cdots \longleftarrow C^{i+1}(K,\R) \stackrel{\delta_{i} }{\longleftarrow} C^i(K,\R) \stackrel{\delta_{i-1} }{\longleftarrow} C^{i-1}(K,\R) \longleftarrow \cdots  {\longleftarrow} C^{-1}(K,\R) \longleftarrow  0,$$
satisfying $\delta_{i}\circ \delta_{i-1}=0$.
The \emph{$i$-th reduced cohomology group}  is defined to be $\tilde{H}^i(K,\R)=\ker \delta_i / \im \delta_{i-1}.$
As vector spaces, $\tilde{H}^i(K,\R)$ is the dual of $\tilde{H}_i(K,\R)$ and is isomorphic to $\tilde{H}_i(K,\R)$.

 By introducing inner products in $C^i(K,\R)$ and $C^{i+1}(K,\R)$ respectively, we have the adjoint $\d^*_i: C^{i+1}(K,\R) \to C^i(K,\R)$ of $\d_i$, which is defined by
$$ ( \delta_i f_1, f_2 )_{C^{i+1}} =( f_1, \delta^*_i f_2)_{C^{i}}$$ for all $f_1 \in C^i(K,\R), f_2 \in C^{i+1}(K,\R)$.

\begin{defi}\cite{HJ13B}
The following three operators are defined on $C^i(K,\R)$, where $K$ is a complex with an orientation $\s$.

(1) The $i$-dimensional combinatorial up Laplace operator or simply the \emph{$i$-up Laplace operator} is defined to be  $  \L_i^{up}(K,\s):=\d_i^* \d_i.$

(2) The $i$-dimensional combinatorial down Laplace operator or the \emph{$i$-down Laplace operator} is defined to be $ \L_i^{down}(K,\s):=\d_{i-1}\d_{i-1}^*.$

(3) The $i$-dimensional combinatorial Laplace operator or the \emph{$i$-Laplace operator} is defined to be $\L_i(K,\s)=\L_i^{up}(K,\s)+\L_i^{down}(K,\s)=\delta^*_i \delta_i+\delta_{i-1}\delta^*_{i-1}.$
 \end{defi}

We will simply use $\L_i^{up}(K)$, $\L_i^{down}(K)$ and $\L_i(K)$ if the orientation $\s$ is clear  from the context.
In fact, the spectra of $\L_i^{up}(K)$ , $\L_i^{down}(K)$ and $\L_i(K)$ are all independent of the orientation; see Lemma \ref{reverse}.
All the three Laplacians are self-adjoint, nonnegative and compact.

Define a weight function on all faces of $K$:
$ w: \cup_{i=-1}^{\dim K} S_i(K) \to \R^+,$
and introduce the inner product in $C^i$ as
$$ (f,g)_{C^i}=\sum_{F \in S_i(K)} w(F) f([F])g([F]).$$
If $w \equiv 1$ on all faces, then the underlying Laplacian is the \emph{combinatorial Laplace operator}, denoted by $L_i(K)$, as discussed in \cite{Duv, Fried}.
If the weights of all facets are equal to $1$, and $w$ satisfies the normalizing condition:
\begin{equation}\label{deg} w(F)=\sum_{\bar{F} \in S_{i+1}(K): F \in \p \bar{F}}w(\bar{F}),\end{equation}
for every $F \in S_i(K)$ which is not a facet of $K$,
then $w$ determines the \emph{normalized Laplace operator}, denoted by $\Delta_i(K)$, as analyzed in \cite{HJ13B}.
The right side of (\ref{deg}) is defined to be the \emph{degree} of the face $F$, denoted by $\deg F$.

 Horak and Jost \cite{HJ13A, HJ13B} give explicit formulas for $\L_i^{up}(K)$ and $\L_i^{down}(K)$ as follows.
\begin{equation}\label{upp}
(\L_i^{up}(K) f)([F])  =\sum_{\bar{F} \in S_{i+1}(K): \atop F \in \p \bar{F}} \frac{w(\bar{F})}{w(F)} f([F]) + \!\!\! \sum_{F' \in S_i(K): F'\ne F, \atop F \cup F'=\bar{F} \in S_{i+1}(K)}\!\!\!
\frac{w(\bar{F})}{w(F)} \sgn([F], \p [\bar{F}]) \sgn([F'], \p [\bar{F}]) f([F']).
\end{equation}
\begin{equation}\label{down}
(\L_i^{down}(K) f)([F])  =\sum_{E \in \p F} \frac{w(F)}{w(E)} f([F]) + \!\!\! \sum_{F' \in S_i(K): F'\ne F, \atop F \cap F'=E \in S_{i-1}(K)}\!\!\!
\frac{w(F')}{w(E)} \sgn([E], \p [F]) \sgn([E], \p [F']) f([F']).
\end{equation}

\subsection{Signed graph and signed incidence graph of a complex}
A \emph{signed graph} $(G,\varsigma)$ is a graph $G$ with a signing on its edges: $\varsigma: E(G) \to \{1,-1\}$.
In particular, if $\varsigma \equiv 1$, namely all edges have positive signs, we will use $(G,+)$ to denote the signed graph.
The \emph{sign} of a subgraph $H$ of $G$, denoted by $\varsigma(H)$, is defined to be the product of the signs of all edges of $H$.
The signed graph $(G,\varsigma)$ is called \emph{balanced} if all signs of the cycles of $G$ are positive.
The signed graph and its balancedness were introduced by Harary \cite{Harary} on research of social psychology.

Switching is an important operation on signed graphs which preserves the balancedness.
A \emph{switching} $s_v$ on $\Gamma$ at a vertex $v$  is an operation which reverses the signs of all edges incident to $v$, and keeps the signs of other edges invariant.
The signed graphs $(G,\varsigma)$ is called \emph{switching equivalent} to a signed graph $(G,\varsigma')$ if $(G,\varsigma')$ can be obtained from $(G,\varsigma)$ by a sequence of switchings on vertices.
It is known that switching does not change the signs of cycles and hence the balancedness of a signed graph.

\begin{lem}\cite[Corollary 3.3]{Z1}\label{balance}
A signed graph $(G,\varsigma)$ is balanced if and only if it can be switched to $(G,+)$.
\end{lem}

Let $K$ be a complex with an orientation $\s$, and let $F, \bar{F} \in K$.
If $F \in \p \bar{F}$, then $(F,\bar{F})$ is an incidence of $K$.
The \emph{$i$-th incidence graph} $B_i(K)$ is a bipartite graph with vertex set $S_i(K)\cup S_{i+1}(K)$ such that $\{F,\bar{F}\}$ is an edge if and only if $F\in\partial\bar{F}$.
Actually $B_i(K)$ is a signed graph such that the sign of an edge $\{F,\bar{F}\}$ is given by  $\sgn([F],\partial[\bar{F}])$, denoted by $(B_i(K),\s)$ and called the \emph{$i$-th signed incidence graph} of $K$.
Note that a \emph{signature matrix} is a diagonal matrix with $\pm 1$ on its diagonal.

The following lemma gives a relation between the reverse of faces of $K$ and the switching equivalence of the signed incidence graph of $K$, and also the change of the Laplacians of $K$ after the reverse of faces.

\begin{lem}\cite{FWW}\label{reverse}
Let $K$ be a complex with an orientation $\s$, and let $(B_i(K),\s)$ denote the corresponding $i$-th signed incidence graph.
Let $\tau$ be another orientation of $K$.
Then the following results hold.

$(1)$ $(B_i(K),\tau)$ is obtained from $(B_i(K),\s)$ only by applying
 a switching $s_F$ on an $i$- or $(i+1)$-face $F$ of $K$
 if and only if $\tau$ is obtained from $\s$ only by reversing
the orientation of the face $F$, where the reversed orientation of a $0$-face $[v]$ is defined to be $-[v]$.

$(2)$ If $\tau$ is obtained from $\s$ only by reversing
the orientation of $i$-face $F$, then
$$ \L_i^{up}(K,\tau)=S_F \L_i^{up}(K,\s)S_F,$$
where $S_F$ is a signature matrix defined on the $i$-th elementary cochains of $K$ with only $-1$ on the diagonal corresponding to $[F]^*$.

$(3)$ If $\tau$ is obtained from $\s$ only by reversing
the orientation of an $(i+1)$-face $F$, then
$$ \L_i^{up}(K,\tau)=\L_i^{up}(K,\s).$$
\end{lem}

By Lemma \ref{reverse}, we know for any two orientations $\s, \tau$ on $K$, $B_i(K,\s)$ is switching equivalent to $B_i(K,\tau)$, and hence they have the same balancedness.

\section{A characterization of the $i$-th normalized Laplacian with eigenvalue $i+2$}

Let $K$ be a complex, and let $\lamax(\Delta_i^{up}(K))$ be the largest eigenvalue of $\Delta_i^{up}(K)$.
Horak and Jost \cite{HJ13B} proved that $\lamax(\Delta_i^{up}(K)) \le i+2$, and give an characterization of the complex $K$ with $\lamax(\Delta_i^{up}(K)) = i+2$ (Theorem \ref{jost}).
In this section, we will use the $i$-th signed incidence graph of $K$ to give an alternative characterization of the equality case.
We remark that by definition $\lamax(\Delta_{i+1}^{down}(K))=\lamax(\Delta_i^{up}(K))$.
So the following result also holds for the $(i+1)$-down normalized Laplacian.
We always assume the complex $K$ has the dimension greater then $i$ when discussing the eigenvalues of $\Delta_i^{up}(K)$.

\begin{thm}\label{main-result}
Let $K$ be a complex with an orientation $\s$.
Then $\lamax(\Delta_i^{up}(K))=i+2$ if and only if $K$ contains an $(i+1)$-path connected components $K'$ such that $(B_i(K'),\s)$ is balanced.
\end{thm}

\begin{proof}
($\Rightarrow$) Let $f$ be an eigenfunction of $\Delta_i^{up}(K)$ associated with the largest eigenvalue $\lamax(\Delta_i^{up}(K))$.
Then we have
\begin{equation}\label{i+2}
\begin{aligned}
(i+2) (f,f)&=(\Delta_i^{up}f,f)=(\delta_i f,\delta_i f)\\
&=\sum_{\bar{F}\in S_{i+1}(K)}f(\partial[\bar{F}])^2\omega(\bar{F})\\
&=\sum_{\bar{F}\in S_{i+1}(K)}\big(f(\sum_{F\in\partial \bar{F}}\text{sgn}([F],\partial[\bar{F}])[F])\big)^2
\omega(\bar{F})\\
&=\sum_{\bar{F}\in S_{i+1}(K)}\big(\sum_{F\in\partial\bar{F}}\text{sgn}([F],\partial[\bar{F}])f([F])\big)^2
\omega(\bar{F})\\
&\le \sum_{\bar{F}\in S_{i+1}(K)} (i+2) \big(\sum_{F\in\partial\bar{F}} f([F])^2\big) \omega(\bar{F})\\
&= (i+2) \sum_{F \in S_i(K)}  f([F])^2 \sum_{\bar{F}: F \in \p \bar{F}} \omega(\bar{F})\\
&=(i+2)\sum_{F\in S_i(K)}f([F])^2\omega(F)=(i+2) (f,f),
\end{aligned}
\end{equation}
where the inequality in (\ref{i+2}) follows from Cauchy-Schwarz inequality.
So the inequality in (\ref{i+2}) holds as an equality, and then 
for each $\bar{F}\in S_{i+1}(K)$,
$$\big(\sum_{F\in\partial\bar{F}}\text{sgn}([F],[\partial[\bar{F}])f([F])\big)^2
=(i+2)\sum_{F\in S_i(K)}f([F])^2.$$
Hence, for each $\bar{F} \in S_{i+1}(K)$, there exist a constant $k_{\bar{F}}$ such that
for all $F\in \partial\bar{F}$,
\begin{equation}\label{kf}
f([F])=k_{\bar{F}}\text{sgn}([F],[\partial[\bar{F}]),
\end{equation}
and then
\begin{equation}\label{kfA}
|f([F])| = |k_{\bar{F}}|.
\end{equation}

Now let $F \in S_i(K)$ such that $f([F])\ne 0$, and let $K'$ be the $(i+1)$-path connected component of $K$ which contains the face $F$.
Then $B_i(K')$ is a connected component of $B_i(K)$.
So, for any two faces $F_1, F_2 \in S_i(K')$, there exists an path connecting them, and hence $|f([F_1])|=|f([F_2])|$ by (\ref{kfA}), which implies that for all $F \in S_i(K')$ and $\bar{F} \in S_{i+1}(K')$,
\begin{equation}\label{kfCom}
|f([F])|=\alpha=|k_{\bar{F}}|>0,
\end{equation}
for some positive constant $\alpha$.

Let $C$ be an cycle of $(B_i(K'),\s)$ on the vertices $F_1, \bar{F}_1, F_2, \bar{F}_2, \ldots, F_k, \bar{F}_k, F_1$, where $F_i,F_{i+1} \in \bar{F}_i$ for $i \in [k]$ and $F_{k+1}=F_1$.
By (\ref{kf}), for $i \in [k]$, we have
\begin{equation}\label{cycle}
f([F_i])=k_{\bar{F}_i} \sgn ([F_i],\p [\bar{F}_i]),
f([F_{i+1}])=k_{\bar{F}_i} \sgn ([F_{i+1}],\p [\bar{F}_i]).
\end{equation}
Now multiplying the equalities in (\ref{cycle}) over all $i \in [k]$, we have
$$ \prod_{i \in [k]} f([F_i])^2 =  \prod_{i \in [k]} k_{\bar{F}_i}^2 \cdot \sgn C,$$
where $\sgn C$ denotes the sign the cycle $C$.
So we have $\sgn C=1$, and then $B_i(K',\s)$ is balanced by definition.

($\Leftarrow$) Let $K'$ be an $(i+1)$-path connected component of $K$ such that $B_i(K',\s)$ is balanced.
Then  $B_i(K',\s)$ can be switched to $B_i(K',+)$ by a sequence of switching on the vertices of $B_i(K')$ (namely, the $i$-faces and/or $(i+1)$-faces of $K'$).
By Lemma \ref{reverse}, there exists an orientation $\tau$ of the $i$-faces and/or $(i+1)$-faces of $K'$ such that  $B_i(K',\tau)$ has all edges with positive signs.
Now let $f$ be an function defined on $(K',\tau)$ such that $f([F])=1$ for all $F \in S_i(K')$ under the orientation of $\tau$.
Now by the formula (\ref{upp}) of $\Delta_i^{up}(K,\tau) f$, we have
\begin{align*}
(\Delta_i^{up}(K,\tau) f)([F]) & =\sum_{\bar{F} \in S_{i+1}(K'): \atop F \in \p \bar{F}} \frac{w(\bar{F})}{w(F)} f([F]) + \!\!\! \sum_{F' \in S_i(K'): F'\ne F, \atop F \cup F'=\bar{F} \in S_{i+1}(K')}\!\!\!
\frac{w(\bar{F})}{w(F)} \sgn([F], \p [\bar{F}]) \sgn([F'], \p [\bar{F}]) f([F'])\\
&=\sum_{\bar{F} \in S_{i+1}(K'): \atop F \in \p \bar{F}} \frac{w(\bar{F})}{w(F)}
+\!\!\! \sum_{F' \in S_i(K'): F'\ne F, \atop F \cup F'=\bar{F} \in S_{i+1}(K')}\!\!\!
\frac{w(\bar{F})}{w(F)} \\
&= \sum_{\bar{F} \in S_{i+1}(K'): \atop F \in \p \bar{F}} \frac{w(\bar{F})}{w(F)}
+ (i+1)\!\!\!\sum_{\bar{F} \in S_{i+1}(K'): F \in \p \bar{F}}
\frac{w(\bar{F})}{w(F)} \\
&=i+2.
\end{align*}
So we have $\Delta_i^{up}(K',\tau) f=(i+2) f$, which implies that $i+2$ is an eigenvalue of
$\Delta_i^{up}(K',\tau)$.
By Lemma \ref{reverse}, the spectrum of $\Delta_i^{up}(K')$ is independent of the orientation.
So $\Delta_i^{up}(K')$ has an eigenvalue $i+2$.
As $\Delta_i^{up}(K')$ is a direct summand of $\Delta_i^{up}(K)$, $\Delta_i^{up}(K)$ has an eigenvalue $i+2$ (namely the largest eigenvalue).
\end{proof}

\begin{cor}\label{multi}
Let $K$ be a complex with an orientation $\s$.
Then the multiplicity of $i+2$ as an eigenvalue of $\Delta_i^{up}(K)$ is equal to the number of $(i+1)$-path connected component $K'$ of $K$ with $B_i(K',\s)$ being balanced.
\end{cor}

\begin{proof}
Let $K_1,\ldots, K_s$ be all $(i+1)$-path connected components of $K$, where $s \ge 1$.
Then, by a relabelling  of $i$-faces of $K$, we have
$$ \Delta_i^{up}(K)=\Delta_i^{up}(K_1)\oplus \cdots \oplus \Delta_i^{up}(K_s) \oplus O,$$
where $O$ is a zero matrix corresponding to those $i$-faces which are not in any $(i+1)$-faces of $K$.
If $i+2$ is an eigenvalue of $\Delta_i^{up}(K)$, then it is an eigenvalue of $\Delta_i^{up}(K_t)$ for some $t$, and $B_i(K_t, \s)$ is balanced by Theorem \ref{main-result} .
Then there exists an orientation $\tau$ of $i$- and/or $(i+1)$-faces of $K_t$ such that
$B_i(K_t, \tau)$ has all edges positive.
By (\ref{kf}) and a similar discussion as (\ref{kfCom}), if $f$ is an eigenfunction of  $\Delta_i^{up}(K_t,\tau)$, then $f$ is constant on all $i$-faces of $K_t$.
So the multiplicity of $i+2$ as an eigenvalue of $\Delta_i^{up}(K_t,\tau)$ is $1$.
On the other side, if  $K_t$ is an $(i+1)$-path connected components of $K$ such that  $B_i(K_t, \s)$ is balanced, also by Theorem \ref{main-result}, $i+2$ is an eigenvalue of $\Delta_i^{up}(K_t,\s)$ with multiplicity $1$.
The result now follows.
\end{proof}

Observe that for a $1$-path connected complex $K$, $B_0(K,\s)$ is balanced if and only if the graph $K^{(1)}$ is bipartite. 
So, an immediately conclusion of Corollary \ref{multi} on graph case is:
the multiplicity of $2$ as a normalized Laplaican eigenvalue of $G$ is the number of connected bipartite components of $G$ (see \cite[Chapter 1]{Chung97}).
Our result also implies Horak and Jost's characterization on the eigenvalue $i+2$ (Theorem \ref{jost}).
We need the following notion on orientable complexes.

\begin{defi}\cite{HJ13B}
Let $K$ be a pure $(i+1)$-dimensional simplicial complex. We say that $K$ is \emph{orientable} if and only if it is possible to assign an orientation to all $(i+1)$-faces of $K$ in such a way that any two simplices that intersect in an $i$-face induce a different orientation on that face.
\end{defi}
%
%

\begin{thm}
For an $(i+1)$-path connected simplicial complex $K$, the following statements are equivalent:

$(1)$ the $i$-th incidence signed graph $B_i(K)$ is balanced;

$(2)$ there are no $(i+1)$-orientable circuits of odd length nor $(i+1)$-non-orientable circuits of even length in $K$.
\end{thm}

\begin{proof}
($\Rightarrow$) Assume to the contrary that there exists an $(i+1)$-orientable circuit $C$ of odd length. By definition, there exists $2s+1$ $(i+1)$-faces $\bar{F}_1, \bar{F}_2,\ldots,\bar{F}_{2s+1}$ such that $F_i:=\bar{F}_i \cap \bar{F}_{i+1} \in S_i(K)$ for $i \in [2s+1]$, where $\bar{F}_{2s+2}=\bar{F_1}$.
Also, there exists an orientation of $K$ such that for all $i\in [2s+1]$,
$$ \sgn([F_i], \p [\bar{F}_i])\sgn ([F_i], \p [\bar{F}_{i+1}])=-1.$$
Now the signed graph $B_i(K,\s)$ contains a cycle $C$:
$$ \bar{F}_1, F_1, \bar{F}_2, F_2, \bar{F}_3, \cdots, \bar{F}_{2s+1}, F_{2s+1}, \bar{F}_1.$$
The sign of $C$ is
$$ \sgn C=\prod_{i=1}^{2s+1} \sgn([F_i], \p [\bar{F}_i])\sgn ([F_i], \p [\bar{F}_{i+1}])
=(-1)^{2s+1}=-1.$$
By definition $B_i(K,\s)$ is not balanced; a contradiction.

If $K$ has an $(i+1)$-non-orientable circuit of even length, say
$\bar{F}_1, \bar{F}_2,\ldots,\bar{F}_{2s}$, where $F_i:=\bar{F}_i \cap \bar{F}_{i+1} \in S_i(K)$ for $i \in [2s]$ ($\bar{F}_{2s+1}=\bar{F_1}$).
There exists an orientation of $K$ such that for all $i\in [2s-1]$,
$$ \sgn([F_i], \p [\bar{F}_i])\sgn ([F_i], \p [\bar{F}_{i+1}])=-1$$
and
$$  \sgn([F_{2s}], \p [\bar{F}_{2s}])\sgn ([F_{2s}], \p [\bar{F}_1])=1.$$
So, the sign of $C$ is
$$ \sgn C=\prod_{i=1}^{2s} \sgn([F_i], \p [\bar{F}_i])\sgn ([F_i], \p [\bar{F}_{i+1}])
=(-1)^{2s-1}=-1,$$
which also yields a contradiction to the balancedness of $B_i(K,\s)$.


($\Leftarrow$) To the contrary assume that $B_i(K,\s)$ is unbalanced.
Then $B_i(K,\s)$ contains a negative cycle
$$ \bar{F}_1, F_1, \bar{F}_2, F_2, \cdots, F_{t-1}, \bar{F}_t, F_t, \bar{F}_1,$$
with
$$ \sgn C=\prod_{i=1}^{t} \sgn([F_i], \p [\bar{F}_i])\sgn ([F_i], \p [\bar{F}_{i+1}])=-1,$$
where $F_i:=\bar{F}_i \cap \bar{F}_{i+1} \in S_i(K)$ for $i \in [t]$ ($\bar{F}_{t+1}=\bar{F_1}$).
We can assign new orientations to $\bar{F}_i$ for $i \in [t]$ such that for
$$\sgn([F_i], \p [\bar{F}_i])\sgn ([F_i], \p [\bar{F}_{i+1}])=-1, ~i\in [t-1].$$
Note that the sign of $C$ is unchanged after the above orientations.
So
$$ \sgn C=(-1)^{t-1} \sgn([F_t], \p [\bar{F}_t])\sgn ([F_t], \p [\bar{F}_1])=-1.$$
Hence, if $t$ is odd, then
$$ \sgn([F_t], \p [\bar{F}_t])\sgn ([F_t], \p [\bar{F}_1])=-1,$$
implying that $C$ is an $(i+1)$-orientable circuit of odd length; a contradiction.
If $t$ is even, then
$$ \sgn([F_t], \p [\bar{F}_t])\sgn ([F_t], \p [\bar{F}_1])=1,$$
implying that $C$ is an $(i+1)$-non-orientable circuit of even length; also a contradiction.
%
%
%
\end{proof}

\section{Complexes with the $i$-th normalized Laplacian having eigenvalue $i+2$}
For an $(i+1)$-path connected complex $K$, we  proved that
the largest eigenvalue of $\Delta_i^{up}(K)$ is $i+2$ if and only if the $i$-th signed incidence graph $B_i(K)$ of $K$ is balanced.
Fan, Wu and Wang \cite{FWW} characterized the balancedness of the signed incidence graph of complexes under operations such as wedge sum, join, Cartesian product and duplication of motifs.
In this section, by using the above operations on complexes we can construct infinitely many complexes $K$ for which the largest eigenvalue of $\Delta_i^{up}(K)$ is $i+2$.

\subsection{Wedge}
\begin{defi}\cite{HJ13B}
Let $K_1$ and $K_2$ be complexes on disjoint vertex sets $V(K_1)$ and $V(K_2)$, respectively.
Let $F_1=\{v_0,\ldots, v_k\} \in S_k(K_1)$ and $F_2=\{u_0,\ldots, u_k\} \in S_k(K_2)$.
The combinational \emph{$k$-wedge sum} of $K_1$ and $K_2$ is a complex obtained from $K_1 \cup K_2$ by identifying  $F_1$ with $F_2$ such that $v_i$ is identified with $u_i$ for $i=0,1,\ldots,k$.
\end{defi}

The above definition generalizes in an obvious way to the $k$-wedge sum of arbitrary many complexes.
If $k \ge 1$, the $k$-wedge sum depends on the ways of identification of vertices. In fact, there exists a bijection $\phi: F_1 \to F_2$ such that $v \in F_1$ is identified with $\phi(v) \in F_2$.

\begin{thm}\label{wedge}\cite{FWW}
Let $K=K_{1}\vee_k K_2$ be a $k$-wedge sum of complexes $K_1$ and $K_2$ by identifying a $k$-face $F_1$ of $K_1$ and a $k$-face $F_2$ of $K_2$.
If $i \ge k-1$, the $i$-th signed incidence graph $B_i(K)$ of $K$ is balanced if and only if $B_i(K_1)$ and $B_i(K_2)$ both are balanced.
\end{thm}
\begin{thm}\label{wedge-eigenvalue}
Let $K=K_{1}\vee_k K_2$ be a $k$-wedge sum of complexes $K_1$ and $K_2$ by identifying a $k$-face $F_1$ of $K_1$ and a $k$-face $F_2$ of $K_2$, where $K_1,K_2$ are both $(i+1)$-path connected. Then the following results hold.

$(1)$ If $i \in \{k-1,k\}$, then $K$ is $(i+1)$-path connected, and $\lamax(\Delta_i^{up}(K))=i+2$ if and only if $B_i(K_1)$ and $B_i(K_2)$ both are balanced.

$(2)$ If $i \ge k+1$, then $K$ is not $(i+1)$-path connected, and $\lamax(\Delta_i^{up}(K))=i+2$ if and only if one of $B_i(K_1)$ and $B_i(K_2)$ is balanced.
\end{thm}

\begin{proof}
Note that $K$ is $(i+1)$-path connected if and only if $B_i(K)$ is connected.
If $i =k$, then $B_i(K)$ is a $0$-wedge sum of $B_i(K_1)$ and $B_i(K_2)$ by identifying $F_1$ with $F_2$.
As $B_i(K_1)$ and $B_i(K_2)$ are both connected, $B_i(K)$ is also connected.
If $i=k-1$, then  $B_i(K)$ is obtained from $B_i(K_1)$ and $B_i(K_2)$  by identifying $F_1$ and $F_2$ together with those $(k-1)$-faces in $\p F_1$ and their corresponding faces in $\p F_2$.
So $B_i(K)$ is also connected in this case.
By Theorem \ref{main-result}, $\lamax(\Delta_i^{up}(K))=i+2$ if and only if $B_i(K)$ is balanced.
The result follows from Theorem \ref{wedge}.

If $i \ge k+1$, then $B_i(K)$ is a disjoint union of $B_i(K_1)$ and $B_i(K_2)$, and hence
$\Delta_i^{up}(K)$ is a direct sum of $\Delta_i^{up}(K_1)$ and $\Delta_i^{up}(K_2)$.
So $\lamax(\Delta_i^{up}(K))=i+2$ if and only if $\lamax(\Delta_i^{up}(K_1))=i+2$ or $\lamax(\Delta_i^{up}(K_2))=i+2$.
The result  follows from Theorem \ref{main-result}.
\end{proof}


Fan, Wu and Wang \cite{FWW} proved that for each integer $i \ge 0$, there exist infinitely many $(i+1)$-path connected acyclic complexes $K$ with $B_i(K)$ being balanced.
Similarly, start from an $(i+1)$-simplex $\s$ and let $K_0=\s$, $K_p=K_{p-1}\vee_i K_0$ for $p \ge 1$. We will get a sequence of complexes $\{K_p\}_{p \ge 0}$.
Then $K_p$ is $(i+1)$-path connected, balanced and acyclic; see \cite[Corollary 4.5]{FWW}.
So we have the following result immediately.

\begin{cor}
For each integer $i \ge 0$, there exist infinitely many $(i+1)$-path connected acyclic complexes $K$ such that $\lamax(\Delta_i^{up}(K))=i+2$.
\end{cor}

\subsection{Cartesian product}
\begin{defi}
Let $K_1$ and $K_2$ be complexes with vertex sets $V(K_1)$ and $V(K_2)$ respectively.
The \emph{Cartesian product} $K_1\Box  K_2$ of $K_1$ and $K_2$ is a complex with  vertex set $V(K_1)\times V(K_2)$, whose faces are $F\times v:=F\times \{v\}=\{(u_0,v),\ldots,(u_s,v)\}$ if  $F=\{u_0,\ldots,u_s\}\in K_1$ and $v\in V(K_2)$, or $u\times F':=\{u\}\times F'=\{(u,v_0),\ldots,(u,v_t)\}$ if $u\in V(K_1)$ and $F'=\{v_0,\ldots,v_t\}\in K_2$.
\end{defi}

An orientation of $F$ in $K_1$ induces an orientation of $F\times v$ in $K_1 \Box K_2$.
If $[F]=[u_0,\ldots,u_s]\in K_1$, then
$[F\times v]:=[(u_0,v),\ldots,(u_s,v)]$.
So
$\sgn([F\times v], \partial[\bar{F}\times v])=\sgn([F],\partial[\bar{F}]).$
Similarly, an orientation of $F'$ in $K_2$ induces an orientation of $u\times F'$.
Given weight functions $w_1$ on $K_1$ and $w_2$ on $K_2$.
We define the weight $w$ on $K_1\Box K_2$ such that for any $F \in K_1$ and $F' \in K_2$,
\begin{equation}\label{Cart-weight}
w(F\times v):=w_1(F), ~ w(u\times F'):=w_2(F').
\end{equation}
This implies a condition on $w_1$ and $w_2$, namely, $w(u,v)=w_1(u)=w_2(v)$ for any vertex $u$ of $K_1$ and $v$ of $K_2$.
\begin{thm}\cite{FWW}\label{cart}
Let $K_1, K_2$ be $(i+1)$-path connected complexes.
Then $B_i(K_1\Box K_2)$ is balanced if and only if $B_i(K_1)$ and $B_i(K_2)$ are both balanced.
\end{thm}
\begin{thm}\label{cart-eigenvalue}
Let $K_1, K_2$ be $(i+1)$-path connected complexes.
Then the following results hold.

$(1)$ If $i=0$, $K_1\Box K_2$ is $1$-path connected, and $\lamax(\Delta_0^{up}(K_1\Box K_2))=2$ if and only if $B_0(K_1)$ and $B_0(K_2)$ are both balanced.

$(2)$ If $i \ge 1$, then $K_1\Box K_2$ is not  $(i+1)$-path connected, and $\lamax(\Delta_i^{up}(K_1\Box K_2))=i+2$ if and only if one of $B_i(K_1)$ and $B_i(K_2)$ is balanced.
\end{thm}

\begin{proof}
If $i=0$, then $(K_1\Box K_2)^{(1)}$ is the usual Cartesian product of graphs $K_1^{(1)}$ and $K_2^{(1)}$.
It is known that $(K_1\Box K_2)^{(1)}$ is connected if $K_1^{(1)}$ and $K_2^{(1)}$ are both connected  \cite{IK}.
So $K_1\Box K_2$ is $1$-path connected, and by Theorem \ref{main-result} $\lamax(\Delta_0^{up}(K_1\Box K_2))=2$ if and only if $K_1\Box K_2$ is balanced.
The result follows from Theorem \ref{cart}.

If $i \ge 1$, then by \cite[Lemma 4.12]{FWW},
$B_i(K_1 \Box K_2)$ is the union of following disjoint graphs:
$|V(K_2)|$ copies of $B_i(K_1)$, and $|V(K_1)|$ copies of $B_i(K_2)$, and
$$
\L_i^{up}(K_1\Box K_2)=\L_i^{up}(K_1)\otimes I_{V(K_2)} \oplus I_{V(K_1)}\otimes \L_i^{up}(K_2),
$$
where $I_{V(K_1)},I_{V(K_2)}$ are the identity matrices defined on $V(K_1),V(K_2)$ respectively.
So, $K_1\Box K_2$ is not  $(i+1)$-path connected, and $\lamax(\Delta_i^{up}(K_1\Box K_2))=i+2$ if and only if $\lamax(\Delta_i^{up}(K_1))=i+2$ or $\lamax(\Delta_i^{up}(K_2))=i+2$.
The result  follows from Theorem \ref{main-result}.
\end{proof}


\subsection{Duplication of motifs}
\begin{defi}\cite{HJ13B}
Let $K$ be a complex and $S$ be a collection of simplices in $K$.
The \emph{closure} of $S$ written $\cl S$ is the smallest subcomplex of $K$ that contains all simplices in $S$.
The \emph{star} of $S$ written $\st S$ is the set of all simplices in $K$ that have a face in $S$.
The \emph{link} of $S$, written $\lk S$,  is defined to be $\cl \st S-\st \cl S$.
\end{defi}

\begin{defi}\cite{HJ13B}
Let $K$ be a complex. The subcomplex $\Sigma$ of $K$ is a \emph{motif} if $\Sigma$ contains all simplices in $K$ on the vertices of $V(\Sigma)$.
\end{defi}

\begin{defi}\cite{HJ13B}\label{motif}
Let $\Sigma$ be a subcomplex of a complex $K$. $\Sigma$ is called a \emph{$i$-motif } if the following conditions hold.

(1) ($\forall F_1,F_2\in \Sigma$), if $F_1,F_2\subset F \in K$, then $F\in \Sigma$.

(2) dim $\lk\Sigma=i$.
\end{defi}

\begin{lem}\cite{HJ13B}\label{dimLK}
If $K$ is an $(i+1)$-path connected complex, then any motif satisfying (1) in Definition \ref{motif} will have a link of dimension $i$.
\end{lem}

\begin{defi}
Tow complexes $K$ and $L$ are isomorphic if there exists a bijection $f: V(K)\rightarrow V(L)$ such that $\{v_0,\ldots, v_k\}\in K$ if and only if $\{f(v_0),\ldots, f(v_k)\}\in L$.
\end{defi}

Now let $\Sigma$ be an $i$-motif of $K$ with
vertex set $V(\Sigma)=\{v_0,\ldots, v_k\}$.
Let $V(\lk\Sigma)=\{u_0,\ldots, u_m\}$.
 According to the definition of link, $V(\Sigma)\cap V(\lk\Sigma)=\emptyset$.
 Define a complex $\Sigma'$ with vertex set $V(\Sigma')=\{v'_0,\ldots, v'_k\}$,
 which is isomorphic to $\Sigma$ by the isomorphism  $f: v_i \mapsto v'_i$.
 Then
$$K^{\Sigma}:=K\cup \{\{v'_{i_0},\ldots, v'_{i_l},u_{j_1},\ldots, u_{j_s}\}\mid\{v_{i_0},\ldots, v_{i_l},u_{j_1},\ldots, u_{j_s}\}\in K\}$$
 is a complex obtained from $K$ by the \emph{duplication of the $i$-motif $\Sigma$}.
 Denote by $K_{\Sigma'}=(K \backslash \st \Sigma) \cup \st \Sigma'$, a subcomplex of $K^\Sigma$ which is isomorphic to $K$ via a map $\bar{f}$ with $\bar{f}|_{V(\Sigma)}=f$ and  $\bar{f}|_{V(K)\backslash V(\Sigma)}=id$.

\begin{thm}\cite{FWW}\label{motif}
Let $K$ be an $(i+1)$-path  connected complex and let $\Sigma$ be  an $i$-motif of $K$.
Then $B_i(K^{\Sigma})$ is balanced if and only if $B_i(K)$ is balanced.
\end{thm}

\begin{thm}\label{motif-eigenvalue}
Let $K$ be an $(i+1)$-path  connected complex and let $\Sigma$ be  an $i$-motif of $K$.
Then $K^{\Sigma}$ is $(i+1)$-path  connected, and  $\lamax(\Delta_i^{up}(K^{\Sigma}))=i+2$ if and only if $B_i(K)$ is balanced.
\end{thm}

\begin{proof}
Let $\Sigma'$ and $K_{\Sigma'}$ be  complexes  defined as the above.
Note that by Lemma \ref{dimLK}, the link $\lk \Sigma$ has dimension $i$.
Noting that $B_i(K)$ is connected, so $B_i(K_{\Sigma'})$ is connected as it is isomorphic to $K$.
We also find $B_i(K)$ and $B_i(K_{\Sigma'})$ share common vertices or the faces of $K\backslash \st \Sigma \supseteq \lk \Sigma$.
So $B_i(K^\Sigma)$ is connected, and $K^{\Sigma}$ is $(i+1)$-path  connected.
The result now follows from Theorem \ref{main-result} and Theorem \ref{motif}.
\end{proof}

By Theorem \ref{motif-eigenvalue}, one can apply the duplication of motifs to construct an infinite family of $(i+1)$-path connected complexes $K$ which the largest eigenvalue of $\Delta_i^{up}(K)$ being $i+2$ for each $i \ge 0$.


\begin{thebibliography}{99}


\bibitem{Bachoc} C. Bachoc, A. Gundert, A. Passuello, The theta number of simplicial complexes, \emph{Israel Journal of Mathematics}, 232: 443-481, 2019.


\bibitem{Bot} O. Bottema, \"Uber die Irrfahrt in einem Strassennetz, \emph{Math. Z.},  39: 137-145, 1935.

\bibitem{Bouc}    S. Bouc, Homologie de certains ensembles de 2-sous-groups des groupes
sym\'etriques, \emph{J. Algebra}, 150: 158-186, 1992.


\bibitem{Chung93} F. Chung, The Laplacian of a hypergraph, \emph{DIMACS Ser. Discrete Math. Theoret. Comput. Sci.}, pp. 21-36, 1993.

\bibitem{Chung97} F. Chung, \emph{Spectral Graph Theory},  CBMS Regional Conference Series in Mathematics 92, Amer. Math. Soc., 1997.





\bibitem{Dong} X. Dong, M. L. Wachs, Combinatorial Laplacian of the Matching Complex, \emph{ Electron. J. Combin.}, 9(1), 2003.

\bibitem{Duv} A. M. Duval, V. Reiner, Shifted simplicial complexes are Laplacian integral, \emph{Trans. Amer. Math. Soc.}, 354(11): 4313-4344, 2002.

\bibitem{Eckmann} B. Eckmann, Harmonische Funktionen und Randwertaufgaben in einem Komplex, \emph{Comment. Math. Helv.}, 17 (1): 240-255,  1944.



\bibitem{FWW} Y.Z. Fan, H.F. Wu, Y. Wang, The largest Laplacian eigenvalue and the balancedness of simplicial complexes, arXiv: 2405.19078.

\bibitem{Fied} M. Fiedler, Algebraic connectivity of graphs, \emph{Czechoslovak Math. J.}, 23(2): 298-305, 1973.

\bibitem{Fried} J. Friedman, Computing betti numbers via combinatorial Laplacians, \emph{Proceedings of the 28th Annual ACM Symposium on Theory of Computing}, pp. 386-391, 1996.

\bibitem{Golubev} K. Golubev, On the chromatic number of a simplicial complex, \emph{Combinatorica}, 37: 953-964, 2017.




\bibitem{Harary}F. Harary, On the notion of balance of signed graph, \emph{Michigan Math. J.},  2(2), 1953.


\bibitem{HJ13A} D. Horak, J. Jost, Interlacing inequalities for eigenvalues of discrete Laplace operators, \emph{Ann. Global Anal. Geom.}, 43: 177-207,  2013.

\bibitem{HJ13B}D. Horak, J. Jost,  Spectra of combinatorial Laplace operators on simplicial complexes, \emph{Adv. in Math.}, 244: 303-336, 2013.


\bibitem{IK} W. Imrich, S. Klav\v{z}ar, \emph{Product Graphs}, Wiley-Interscience, New York, 2000.


\bibitem{Kirch} G. Kirchhoff, Uber die Aufl\"osung der Gleichungen, auf welche man bei der untersuchung der linearen verteilung galvanischer Str\"ome gef\"uhrt wird, \emph{Ann. Phys. Chem.}, 72: 497-508, 1847.










\bibitem{Z1} T. Zaslavsky, Signed graphs, \emph{Discrete Appl. Math.}, 4: 47-74, 1982.

\end{thebibliography}
\end{document}